\documentclass[a4paper]{amsart}
\usepackage{amsmath,amsthm}
\usepackage{amssymb}
\usepackage{amsfonts}
\usepackage{latexsym}
\usepackage{amsxtra}
\usepackage{subfigure}
\usepackage{graphicx}
\usepackage{pstricks,pst-node,pst-text,pst-3d, pst-plot, pst-grad}
\usepackage{euler}

\numberwithin{equation}{section}

\newtheorem{definition}{Definition}[section]
\newtheorem{theorem}{Theorem}[section]
\newtheorem{corollary}[theorem]{Corollary}
\newtheorem{lemma}[theorem]{Lemma}
\newtheorem{example}[theorem]{Example}
\newtheorem{proposition}[theorem]{Proposition}
\newtheorem{remark}[theorem]{Remark}

\DeclareMathOperator{\J}{\mathrm{Jac}}
\newcommand{\cl}[1]{\overline{#1}}
\newcommand{\R}{\mathbb{R}}

\newcommand{\cL}{\mathcal{L}}
\newcommand{\sign}{\mathop\mathrm{sign}\nolimits}

\newcommand{\ud}{\;\mathrm{d}}
\newcommand{\inter}{\mathrm{int}}
\newcommand{\abs}[1]{\left| #1 \right|}
\newcommand{\norm}[1]{\left\| #1 \right\|}
\newcommand{\scal}[2]{\langle #1 \,,\, #2 \rangle}
\newcommand{\Span}{\mathrm{span}}

\author{Laura Poggiolini}
\author{Marco Spadini}
\address[L.\ Poggiolini, M.\ Spadini]{Dipartimento di Sistemi e Informatica, Universit\`a 
di Firenze, Via Santa Marta 3, 50139 Firenze, Italy}
\title[Local inversion of planar maps with\ldots]{Local inversion of planar maps 
with nice nondifferentiability structure}

\begin{document}

\maketitle

\section{Introduction}

Inspired by invertibility problems for $PC^1$ maps (see e.g.,
\cite{KS94}) that naturally arise in Optimal Control (see e.g.,
\cite{PSp10}) we focus on the invertibility of continuous maps of the
plane which are piecewise linear.

When the plane is pie-sliced in $n\leq 4$ parts (with nonempty
interior and common vertex at the origin) our main result, Theorem
\ref{th.main} below, provides a sufficient condition for any map $L$,
that is continuous and piecewise linear relatively to this slicing, to
be invertible. Some examples show that the assumptions of the theorem
cannot be relaxed too much. In particular, convexity of the slices
cannot be dropped altogether when $n=4$ and, perhaps not
surprisingly, this result cannot be plainly extended to a greater
number of slices.  This result is proved by a combination of linear
algebra and topological arguments in which Theorems 4 and 5 of
\cite{PR96} (Theorems \ref{th.4pr96} and \ref{th.5pr96} below) play a
crucial role.  By contrast, an important tool of nonsmooth analysis,
Clarke's Theorem \cite{Cla83}, does not appear to be adequate for our
purposes in the case $n=4$. We exhibit an explicit example that shows
how this case cannot be treated completely by Clarke's Theorem.

Our results depend on the particularly nice nondifferentiability
structure that we assume throughout. In fact example 2.1 in \cite{KS94}
shows that there exists a $PC^1$ function with $4$ selection functions
(which does not have such structure) which is not locally invertible
at the origin despite being Fr\'echet differentiable at $0$ with
invertible differential.

As stated above, our interest in the invertibility of $PC^1$ maps
stems from optimal control problems. Namely, if one considers a
multiinput optimal control problem which is affine with respect to the
control variable $u\in [-1,1]^m$, then one cannot exclude the
existence of bang--bang Pontryagin extremals. This gives rise to a
$PC^1$ maximized Hamiltonian flow. In order to prove the optimality of
the given Pontryagin extrema via Hamiltonian methods, one needs to
prove the invertibility of the projection of such flow on the state
space (see \cite{AS04} for an introduction to Hamiltonian methods in
control and \cite{ASZ02b,PS04} for specific applications to bang--bang
Pontryagin extremals). In particular, as in \cite{PSp08,PSp10} we are
interested in what happens when two control components switch
simultaneously just once. In this case the ``interesting'' part of the
above-mentioned projection is $2$-dimensional. This justifies our
interest into the invertibility of planar maps. Moreover, a double
switch gives rise to the ``nice'' nondifferentiability structure we
consider in this paper with at most $n=5$ pie-slices which reduce to
$4$ for the subsequent simple switches.

To the best of our knowledge, a comprehensive treatment of
invertibility results in simple cases is not available in the
literature. This has, perhaps, slowed down the study of bang--bang
Pontryagin extremals with multiple switch behavior.

Some comments are in order concerning some of the illustrations
included in this paper.  Figures \ref{fig.ninv5p}, \ref{fig.ninv4p}
and \ref{fig.noclarke4p} represent the piecewise linear maps contained
in Examples \ref{ex.5pezzi}, \ref{ex.ninv4p} and \ref{ex.5pezzicl},
respectively. In fact, they actually show the image of the unit circle
$S^1$ under these maps. But, for the sake of clarity, we have altered
the proportion between axes and, in order to enhance the view close to
the origin, we logarithmically rescaled the radial distance from the
origin. Notice that such transformations do not change the qualitative
behavior of the maps (at least not the characteristics we are
interested in).

\section{Preliminaries and notation}

\subsection{Some notions of nonsmooth analysis}
Following \cite{KS94}, a continuous function $f\colon
U\subseteq\R^s\to\R^m$ is a \emph{continuous selection of $C^1$
  functions} if there exists a finite number of $C^1$ functions
$f_1,\ldots,f_\ell$, of $U$ into $\R^m$ such that the \emph{active
  index set} $\mathcal{I}:=\{i:f(x)=f_i(x)\}$ is nonempty for each
$x\in U$. The functions $f_i$'s are called \emph{selection functions}
of $f$. The function $f$ is called a \emph{$PC^1$ function} if at
every point $x\in U$ there exists a neighborhood $V$ such that the
restriction of $f$ to $V$ is a continuous selection of $C^1$
functions.

A function $f\colon\R^s\to\R^m$ is said to be \emph{piecewise linear}
if it is a continuous selection of linear functions. We will actually
focus on a much more restrictive class of piecewise linear functions
namely in the case $m=s=2$.

\begin{definition}
  A cone with nonempty interior $C$ and vertex at the origin of $\R^k$
  is called a \emph{polyhedral cone} if it is the intersection of a
  finite number of half-spaces.
\end{definition}

\begin{definition}\label{def.pwlin}
  We say that a continuous map $G:\R^k\to\R^k$ is \emph{strongly
    piecewise linear (at $0$)} if there exist a decomposition
  $C_1,\ldots,C_n$ of $\R^k$ in closed polyhedral cones with nonempty
  interior and common vertex at the origin, and linear maps
  $L_1,\ldots,L_n$ with
  \[
  G(x) = L_i x, \qquad x \in C_i.
  \]
  We also say that $G$ is \emph{nondegenerate} if $\sign(\det L_i) $
  is constant and nonzero for all $i=1,\ldots,n$.
\end{definition}

Notice that if $G$ is a continuous strongly piecewise linear map as in
Definition \ref{def.pwlin} above, then $L_i x = L_j x$ for any $x \in
C_i \cap C_j$ and $i,j\in\{1,\ldots,n\}$. Moreover, $G$ is positively
homogeneous.

In this paper we are concerned with the global invertibility of
continuous nondegenerate strongly piecewise linear maps. In this
regard the following simple observation is in order:
\begin{lemma}\label{lem.locglob}
  Let $G:\R^k\to\R^k$ be a continuous strongly piecewise linear map as
  in Definition \ref{def.pwlin}, and let $U$ be an open neighborhood
  of $0\in\R^k$. Assume that the restriction $G|_U:U\to G(U)$ is
  invertible with continuous inverse, then $G$ is globally invertible
  and its inverse is a continuous strongly piecewise linear map as
  well.
\end{lemma}
\begin{proof}
  Let us first prove that $G$ is injective. Let $x_1,x_2\in\R^k$ be
  such that $G(x_1)=G(x_2)$. Let $\rho>0$ be such that the sphere
  $S_\rho$ of radius $\rho$ and centered at the origin is contained in
  $U$. Then
  \[
  G\left(\rho\frac{x_1}{\|x_1\|}\right) =
  G\left(\rho\frac{x_2}{\|x_2\|}\right).
  \]
  Since, for $i=1,2$, $ \rho x_i/\|x_i\|\in U$, we get $x_1=x_2$.

  Let us now prove surjectivity by explicitly exhibiting the
  inverse. This will take care of the continuity too. Given
  $y\in\R^k$, define $H(y)$ as follows:
  \[
  H(y):= \frac{\|y\|}{\rho}
  (G|_U)^{-1}\left(\rho\frac{y}{\|y\|}\right)
  \]
  where $\rho$ is as above. Clearly the above definition does not
  depend on the choice of $\rho$. The fact that $G\big( H(y)\big)=y$
  for any $y\in\R^k$ is a straightforward computation.
\end{proof}

In this paper, we study the invertibility of continuous strongly
piecewise linear maps. We will prove later (Proposition \ref{prop.inv}
below) that, if such a map is invertible, then it is necessarily
nondegenerate. It is not difficult to see that the converse of this
statement is not true (see for instance Examples \ref{ex.5pezzi} and
\ref{ex.ninv4p} below). Our main concern will be finding simple
sufficient conditions for the invertibility. Section \ref{sez.main} is
devoted to this purpose. Before dealing with this problem, however, we
need some preliminaries.

\smallskip A classical notion which we need is that of Bouligand
derivative. Let $U\subseteq\R^s$ be open and let $f\colon U\to\R^m$ be
locally Lipschitz. We say that $f$ is \emph{Bouligand differentiable}
at $x_0\in U$ if there exists a positively homogeneous function,
$f'(x_0,\cdot)\colon\R^s\to\R^m$ with the property that
\begin{equation}\label{defB}
  \lim_{x\to x_0}\frac{\|f(x)-f(x_0)-f'(x_0,x-x_0)\|}{\|x-x_0\|} = 0.
\end{equation}
This uniquely determined function $f'(x_0,\cdot)$ is called the
\emph{Bouligand derivative} of $f$ at $x_0$.
An important fact proved by Kuntz/Scholtes \cite{KS94} is the
following:
\begin{proposition}[Prop.\ 2.1 in \cite{KS94}]
  Let $U\subseteq\R^s$ be an open set. A $PC^1$ function $f\colon
  U\to\R^m$ is locally Lipschitz and, at every $x_0\in U$, has a
  piecewise linear Bouligand derivative $f'(x_0,\cdot)$ which is a
  continuous selection of the Fr\'echet derivatives of the selection
  functions of $f$ at $x_0$.
\end{proposition}

Following \cite{PR96} we consider a generalization of the notion of
Jacobian matrix $\nabla f(x)$ of a function $f\colon\R^k\to\R^k$ at a
Fr\'echet differentiability point $x$. Let $f\colon\R^k\to\R^k$ be
locally Lipschitz at $x_0$. We define $\J (f,x)$ as the (nonempty) set
of limit points of sequences $\{\nabla f(x_k)\}$ where $\{x_k\}$ is a
sequence converging to $x_0$ and such that $f$ is Fr\'echet
differentiable at $x_k$ with Jacobian $\nabla f(x_k)$. One can see
(\cite{PR96}), as a consequence of Rademacher's Theorem that
$\J(f,x_0)$ is nonempty. Moreover the convex hull of $\J(f,x_0)$ is
equal to the Clarke generalized Jacobian of $f$ at $x$.

Let $f\colon U\subseteq\R^k\to\R^k$ be a $PC^1$ function (with
selection functions $f_i$).  The relation between the Bouligand
derivative and the above generalized notion of Jacobian is clarified
by the following formula \cite[Lemma 2]{PR96}:
\begin{equation}\label{eq.genjac}
  \J\big(f'(x_0,\cdot),0\big)\subseteq\J(f,x_0)
  =\big\{\nabla f_i(x_0): i\in\bar{\mathcal{I}}(x_0)\big\},
\end{equation}
where $\bar{\mathcal{I}}(x_0) =\big\{ i:x_0\in \mathrm{cl~int}\{x\in
U:i\in\mathcal{I}(x)\} \big\}$, see e.g.\ \cite{KS94}.

The following two results of \cite{PR96} play a crucial role in the
following. Here, we slightly reformulate them to match our notation.

\begin{theorem}[Thm.\ 4 of \cite{PR96}]\label{th.4pr96}
  Let $f\colon U\subseteq\R^k\to\R^k$ be a $PC^1$ function. Then $f$
  is a local Lipschitz homeomorphism at $x_0\in U$ if and only if
  $\J(f,x_0)$ consists of matrices whose determinants have the same
  nonzero sign and, for a sufficiently small neighbourhood $U_0$ of
  $x_0$, $\deg(f,U_0,0)$ is well-defined and has value $\pm 1$.
\end{theorem}

\begin{theorem}[Thm.\ 5 of \cite{PR96}]\label{th.5pr96}
  Let $f\colon U\subseteq\R^k\to\R^k$ be a $PC^1$ function, and let
  $x_0\in U$. Assume that
  \[
  \J(f,x_0) = J\big(f'(x_0,\cdot),0\big),
  \]
  then the following statements are equivalent:
  \begin{enumerate}
  \item $f$ is a local Lipschitz homeomorphism at $x_0\in U$;
  \item $f'(x_0,\cdot)$ is bijective;
  \item $f'(x_0,\cdot)$ is a (global) Lipschitz homeomorphism.
  \end{enumerate}
  Moreover, if any of (i)--(iii) holds, then $f$ is a local $PC^1$
  homeomorphism at $x_0$.
\end{theorem}

We conclude this subsection recalling the classical notion of
Bouligand tangent cone. Let $C\subseteq\R^k$ be a nonempty closed
subset. Given $x\in C$, the \emph{Bouligand tangent cone to $C$ at
  $x$} is the set:
\[
\big\{ v\in\R^k:\exists\alpha_j\to 0^+,\,\exists v_j\to v\, \textrm{
  s.t.\ } x+\alpha_jv_j\in C\big\}.
\]

\subsection{Topological degree}

In this section we briefly recall the notion of Brouwer degree of a
map and summarize some of its properties that will be used in the rest
of the paper. Major references for this topic are, for instance,
Milnor \cite{Mil65}, Deimling \cite{De95} and Lloyd \cite{Llo}; see
also \cite{BFPS03} for a quick introduction.

A triple $(f,U,p)$, with $p\in\R^k$ and $f$ a proper map defined in
some neighbourhood of the open set $U\subseteq\R^k$, is said to be
\emph{admissible} if $f^{-1}(p)\cap U$ is compact. Given an admissible
triple $(f,U,p)$, it is defined an integer $\deg(f,U,p)$, called the
\emph{degree of $f$ in $U$ respect to $p$}, that in some sense counts
(algebraically) the elements of $f^{-1}(p)$ which lie in $U$. In fact,
when in addition to the admissibility of $(f,U,p)$ we let $f$ be $C^1$
in a neighbourhood of $f^{-1}(p)\cap U$ and assume $p$ is a regular
value of $f$, the set $f^{-1}(p)\cap U$ is finite, and one has
\begin{equation}\label{defgrado}
  \deg(f,U,p)=\sum_{x\in f^{-1}(p)\cap U}\sign\det\big(f'(x)\big),
\end{equation}
where $f'(x)$ denotes the (Fr\'echet) derivative of $f$ at $x$. See
e.g.\ \cite{Mil65} for a broader definition in the case when $(f,U,p)$
is just an admissible triple.

The Brouwer degree enjoys many known properties only a few of which
are needed in this paper. We now remind some of them.

%
%
%
%

\noindent
\emph{(Excision.)} If $(f,U,y)$ is admissible and $V$ is an open
subset of $U$ such that \mbox{$f^{-1}(y) \cap U \subseteq V$}, then
$(f,V,y)$ is admissible and
\[
\deg(f,U,y) = \deg(f,V,y).
\]

\noindent
\emph{(Boundary Dependence.)} Let $U \subseteq \R^k$ be open, and let
$f$ and $g$ be $\R^k$-valued functions defined in a neighbourhood of
$U$ be such that $f(x) = g(x)$ for all $x\in\partial U$. Assume that
$U$ is bounded or, more generally, that $f$ and $g$ are proper and the
difference map $f-g: \cl U \to \R^k$ has bounded image. Then
\[
\deg(f,U,y) = \deg(g,U,y)
\]
for any $y \in \R^k \setminus f(\partial U)$.

%

\medskip Observe that if $f\colon\R^k\to\R^k$ is proper then
$\deg(f,\R^k,p)$ is well-defined for any $p\in\R^k$, moreover, by the
above property, it is actually independent of the choice of $p$. In
this case we shall simply write $\deg(f)$ instead of the more
cumbersome $\deg(f,\R^k,p)$.
%

Finally, we mention a well-known integral formula for the computation
of the degree of an admissible triple when the dimension of the space
is $k=2$ (see e.g.\ \cite{De95,Llo}) which we present here in a
simplified form.

Assume that $f\colon\R^2\to\R^2$ is a proper map, let
$B_r\subseteq\R^2$ be a ball of radius $r>0$ centered at the origin
and let $S_r=rS^1=\partial B_r$. If $0\notin f(S_r)$, then the degree of
$f$ in $B_r$ relative to $0$ coincides with the winding number of the 
curve $\sigma\colon [0,1]\to\R^2$ given by 
\[
 \sigma(t)=f\big(r\cos(2\pi t),r\sin(2\pi t)\big).
\]
In other words,
\[
\deg(f,B_r,0)= \frac{1}{2\pi}\int_{f(S_r)} \omega
\]
where $\omega$ is the $1$-form 
\[
 \omega=\frac{x\ud y}{x^2+y^2} - \frac{y\ud x}{x^2+y^2}
\]
In fact, if $B_r$ is large enough to contain the compact set $f^{-1}(0)$, then
\begin{equation}\label{fintgen}
  \deg(f)= \frac{1}{2\pi}\int_{f(S_r)} \omega.
\end{equation}

\section{Piecewise continuous linear maps and topological degree}

Observe that any nondegenerate continuous strongly piecewise linear
map $G$ is differentiable in $\R^k\setminus\cup_{i=1}^n\partial C_i$.
It is easily shown that $G$ is proper, and therefore $\deg(G,\R^k,p)$
is well-defined for any $p\in\R^k$. In fact, one immediately checks
that $G^{-1}(0)=\{0\}$. So, as remarked above, we can write $\deg(G)$
in lieu of $\deg(G,\R^k,p)$.

The following linear algebra result plays an important role in the
paper.

\begin{proposition}\label{lem:lemalg}
  Let $A$ and $B$ be linear automorphisms of $\R^k$. Assume that for
  some $v\in\R^k\setminus\{0\}$, $A$ and $B$ coincide on the space $\{
  v\}^\perp$.  Then, the map $\mathcal{L}_{AB}$ defined by $x\mapsto
  Ax$ if $ \scal{v}{x} \geq 0$, and by $x\mapsto Bx$ if $ \scal{v}{x}
  \leq 0$, is a homeomorphism if and only if $\det(A)\cdot\det(B)>0$.
\end{proposition}

\begin{proof}
  Let $w_1, \ldots, w_{n-1}$ be a basis of the hyperplane
  $\{v\}^\perp$, then $w_1, \ldots, w_{n-1},v$ is a basis of
  $\R^n$. The matrix of $A^{-1}B$ in this basis is given by
  \[
  \left(
    \begin{array}[c]{c | c}
      \boldsymbol{I}_{n-1} &
      \begin{matrix}
        \gamma_1 \\
        \vdots \\
        \gamma_{n-1}
      \end{matrix} \\[7mm]
      \hline \rule{0pt}{4mm} \boldsymbol{0}^t_{n-1} & \gamma_n
    \end{array}\right)
  \]
  where $ \boldsymbol{I}_{n-1}$ is the $n-1$ unit matrix,
  $\boldsymbol{0}_{n-1}$ is the $n-1$ null vector and the $\gamma_i$'s
  are defined by
  \[
  A^{-1}Bv = \sum_{i=1}^{n-1}\gamma_i w_i + \gamma_n v.
  \]
  Thus $\gamma_n $ is positive if and only if $\det(A)\cdot\det(B)$ is
  positive.

  Observe that if $\gamma_n$ is negative then $\cL_{AB}$ is not
  one--to--one. In fact, being
  \[
  Aw_i = Bw_i,\, \forall i=1,\ldots,n-1,\;\text{ and }\;
  \scal{\sum_{i=1}^{n-1}-\dfrac{\gamma_i}{\gamma_n}\,w_i+\frac{1}{\gamma_n}\,v}{v}=
  \dfrac{\|v\|^2}{\gamma_n}<0,
  \]
  we get
  \begin{multline*}
    \cL_{AB}(v) =A\left( \sum_{i=1}^{n-1}-\dfrac{\gamma_i}{\gamma_n}\,
      w_i +\frac{1}{\gamma_n} \, A^{-1}Bv\right)
    = \sum_{i=1}^{n-1}-\dfrac{\gamma_i}{\gamma_n}\,A w_i +\frac{1}{\gamma_n} \, Bv\\
    =B\left( \sum_{i=1}^{n-1}-\dfrac{\gamma_i}{\gamma_n}\, w_i
      +\frac{1}{\gamma_n} \, v\right) =\cL_{AB}\left(
      \sum_{i=1}^{n-1}-\dfrac{\gamma_i}{\gamma_n}\, w_i +
      \frac{1}{\gamma_n} \, v\right).
  \end{multline*}
 
  We now prove that $\cL_{AB}$ is injective if $\gamma_n$ is
  positive. Assume this is not true. Since both $A$ and $B$ are
  invertible, there exist $z_A, z_B \in\R^n$ such that $\scal{v}{z_A}
  > 0$, $\scal{v}{z_B} < 0$ and $A z_A = B z_B$ or, equivalently,
  $A^{-1}B z_B = z_A$. Let
  \begin{equation*}
    z_A = \sum_{i=1}^{n-1}c^i_A w_i + c_A v, \qquad  z_B =
    \sum_{i=1}^{n-1}c^i_B w_i + c_B v .
  \end{equation*}
  Clearly $c_A >0$, $c_B < 0$. The equality $A^{-1}B z_B = z_A$ is
  equivalent to
  \[
  \sum_{i=1}^{n-1}c^i_B w_i + c_B \sum_{i=1}^{n-1}\gamma_i w_i + c_B
  \gamma_n v = \sum_{i=1}^{n-1}c^i_A w_i + c_A v.
  \]
  Consider the scalar product with $v$, we get $ c_B \gamma_n
  \norm{v}^2 = c_A\norm{v}^2 $, which is a contradiction.

  We finally prove that, if $\gamma_n$ is positive, then $\cL_{AB}$ is
  surjective.  Let $z \in \R^n$. There exist $y_A$, $y_B \in \R^n$
  such that $A y_A = B y_B = z$. If either $\scal{v}{y_A} \geq 0$ or
  $\scal{v}{y_B} \leq 0$, there is nothing to prove.  Let us assume
  $\scal{v}{y_A} < 0$ and $\scal{v}{y_B} > 0$.  In this case $A^{-1}B
  y_B = y_A$ and proceeding as above we get a contradiction.
\end{proof}

\begin{corollary}\label{c.grado2l}
  Let $A,B$ and $v$ be as in Proposition \ref{lem:lemalg}. Define
  $\mathcal{L}_{AB}$, as in Proposition \ref{lem:lemalg}, by
  \[
  \mathcal{L}_{AB}(x)=\begin{cases}
    Ax &\text{if $ \scal{v}{x} \geq 0$,}\\
    Bx &\text{if $ \scal{v}{x} \leq 0$.}
  \end{cases}
  \]
  Assume that $\det(A)\cdot\det(B)>0$. Then
  $\deg(\mathcal{L}_{AB})=\sign\det(A)=\sign\det(B)$.
\end{corollary}
\begin{proof}
  The map $\mathcal{L}_{AB}$ is invertible by Proposition
  \ref{lem:lemalg}. Take any $p\in\R^k$ such that the singleton
  $\{q\}=\mathcal{L}_{AB}^{-1}(p)$ does not belong to 
  $v^\perp$. Then, Formula \ref{defgrado} yields the assertion.
\end{proof}

Another useful tool for the computation of the topological degree of a
strongly piecewise linear map is the following lemma:

\begin{lemma}\label{thm.1}
  If $G$ is a continuous strongly piecewise linear map as in
  Definition \ref{def.pwlin} with $\det(L_i)>0$, $\forall i=1, \ldots,
  n$, then $\deg (G)>0$. In particular, if there exists $q \neq 0$
  whose preimage $G^{-1}(q)$ is a singleton that belongs to at most
  two of the convex polyhedral cones $C_i$, then $\deg(G)=1$.
\end{lemma}

\begin{proof}
  Let us assume in addition that $q\notin\cup_{i=1}^nG\big(\partial
  C_i\big)$. Observe that the set $\cup_{i=1}^nG\big(\partial
  C_i\big)$ is nowhere dense hence $A:=G(C_1)\setminus\cup_{i=1}^n
  G\big(\partial C_i \big)$ is non-empty. Take $x\in A$ and observe
  that if $y\in G^{-1}(x)$ then $y \notin\cup_{i=1}^n \partial
  C_i$. Thus, by \eqref{defgrado},
  \begin{equation}\label{eq.1}
    \deg(G)=\sum_{y\in G^{-1}(x)}\mathrm{sign}\,\det G'(y)
    =\# G^{-1}(x).
  \end{equation}
  Since $G^{-1}(x)\neq\emptyset$, $\deg(G)>0$.

  We now consider the second part of the assertion. Assume in addition
  that $q\notin\cup_{i=1}^nG\big(\partial C_i\big)$.  Taking $x=q$ in
  \eqref{eq.1} we get $\deg(G)=1$.

  Let us now remove the additional assumption. Let $\{p\}=G^{-1}(q)$
  be such that $p\in\partial C_i\cap\partial C_j$ for some $i\neq
  j$. Observe that by assumption $p\neq 0$ does not belong to any cone
  $\partial C_s$ for $s\notin\{ i,j\}$. Thus one can find a
  neighborhood $V$ of $p$, with $V\subset \inter (C_i\cup C_j
  \setminus\{0\})$. By the excision property of the topological degree
  $\deg(G)=\deg(G,V,p)$. Let $\mathcal{L}_{L_iL_j}$ be a map as in
  Proposition \ref{lem:lemalg}.  Observe that, by Corollary
  \ref{c.grado2l}, the assumption on the signs of the determinants of
  $L_i$ and $L_j$ imply that $\deg(\mathcal{L}_{L_iL_j})=1$.  Also
  notice that $\mathcal{L}_{L_iL_j}|_{\partial V}=G|_{\partial
    V}$. Hence, by the excision and boundary dependence properties of
  the degree we have
  \[
  1=\deg(\mathcal{L}_{L_iL_j})=\deg(\mathcal{L}_{L_iL_j},V,p)=\deg(G,V,p).
  \]
  Thus, $\deg(G)=1$ as claimed.
\end{proof}

\begin{remark}\label{rem:th1}
  One can show that if $\det(L_i)<0$, for all $i=1, \ldots, n$, then
  \[
  \deg (G)<0.
  \]
  In particular, if there exists $q \neq 0$ whose preimage $G^{-1}(q)$
  is a singleton that belongs to at most two of the convex cones
  $C_i$, then $\deg(G)=-1$. To see this, it is enough to compose $G$
  with the permutation matrix
  \[
  P=\begin{pmatrix}
    J  & 0\\
    0 & \mathbf{I}_{n-2}
  \end{pmatrix},\quad J:=\begin{pmatrix}
    0  & 1\\
    1 & 0
  \end{pmatrix},
  \]
  and $\mathbf{I}_{n-2}$ is the $(n-2)\times(n-2)$ identity matrix.
\end{remark}

We conclude this section by observing that if $G$ is a nondegenerate
continuous strongly piecewise linear map in $\R^2$ then, by
\eqref{fintgen},
\begin{equation}\label{fintegrale}
  \deg(G)= \frac{1}{2\pi}\int_{G(S^1)} \omega.
\end{equation}
(observe, in fact, that $G^{-1}(0)=\{0\}$). This formula plays an
important role in what follows.

\section{Main results: invertibility of piecewise linear
  maps}\label{sez.main}

We now turn to our main scope that is invertibility of continuous
strongly piecewise linear maps.  We begin with a relatively simple
result.

\begin{proposition}\label{prop.inv}
  Let $G$ be continuous strongly piecewise linear. If $G$ is
  invertible, then it is nondegenerate.
\end{proposition}
\begin{proof}
  Let $C_i$, $i=1, \dots, n$ be the polyhedral cones decomposition of
  $\R^n$ relative to $G$ and let $L_i = \left. G \right\vert_{C_i}$.
  We need to show that $\det(L_i) \neq 0$ for any $i=1, \ldots, n$ and
  that all these determinants have the same sign.

  We first prove that no such determinant is null.  Assume by
  contradiction that, for some $i\in\{1,\ldots,n\}$, $\det(L_i) =
  0$. Without loss of generality we may assume $i=1$. Let $v \in
  \ker(L_1)\setminus\{0\}$. If $v \in C_1$, then $G(v) = L_1 v = 0 =
  G(0)$, so that $G$ is not injective. A contradiction.  If $v \notin
  C_1$, then there exist $w\in \inter(C_1)$ and $\lambda \in \R$,
  $\lambda \neq 0$ such that $w + \lambda v \in \inter(C_1)$. Thus
  $G(w + \lambda v) = L_1(w + \lambda v) = L_1 w = G(w)$, so that $G$
  is, also in this case, not injective.  These contradictions show
  that the determinants $\det(L_i)$'s cannot be zero.

  We now show that all these determinants have the same sign. As in
  the first part of the proof, we proceed by contradiction. Let
  \[
  S := \left\{ C_i \cap C_j \colon i, j \in \{1, \ldots, n\}, \
    \mathrm{codim}\;\Span{ (C_i \cap C_j)} \geq 2 \right\}.
  \]
  Notice that when the dimension $k$ of the ambient space $\R^k$ is $1$, then
  $S=\emptyset$, and if $k=2$ then $S$ is merely the origin. Assume
  by contradiction that there are $i,j\in\{1,\ldots,n\}$ such that
  $\det(L_i)\det(L_j)<0$. Since $\R^k \setminus S$ is arcwise
  connected, it is not difficult to prove that, there must exist two
  cones $C_{\overline{i}}$ and $C_{\overline{j}}$ such that
  $\mathrm{codim}\;\Span(C_{\overline{i}} \cap C_{\overline{j}}) = 1$
  and $\det(L_{\overline{i}}) \det(L_{\overline{j}}) < 0$. Without any
  loss of generality we may assume $\overline i=1$, $\overline j =
  2$. Let $v\in \R^k$ such that $\Span(C_1 \cap C_2) = v^\perp$, $C_1
  \subset \left\{x \in \R^k \colon \scal{v}{x} \geq 0 \right\}$, $C_2
  \subset \left\{x \in \R^k \colon \scal{v}{x} \leq 0 \right\}$. Let
  $w_1, w_2, \ldots, w_{n-1}$ be a basis for $\Span(C_1 \cap C_2)$
  such that
  \[
  \left\{ \displaystyle\sum_{i=1}^{n-1}c_i w_i \colon c_i \geq 0 \quad
    i=1, \ldots, n-1 \right\} \subseteq ( C_1 \cap C_2 ),
  \]
  and let
  \[
  L_1^{-1}L_2v = \gamma_n v + \displaystyle\sum_{i=1}^{n-1}\gamma_i
  w_i .
  \]
  As in the proof of Proposition 3.1 one can show that
  $\gamma_n<0$. Take $c_1, \ldots , c_{n-1} > 0$ and define
  \[
  z_1 := v + \displaystyle\sum_{i=1}^{n-1}c_i w_i\quad\text{and}\quad
  z_2 := \dfrac{1}{\gamma_n} v + \displaystyle\sum_{i=1}^{n-1}\left(
    c_i - \dfrac{\gamma_i}{\gamma_n} \right) w_i.
  \]
  An easy computation shows that $L_1 z_1 = L_2 z_2$. Choosing $c_1,
  \ldots, c_{n-1}$ large enough, we can assume that $z_1 \in C_1$,
  $z_2 \in C_2$. Thus $G(z_1) = G(z_2)$, i.e.\ $G$ is not injective,
  against the assumption. This contradiction shows that all
  determinants $\det(L_s)$, $s\in\{1,\dots,n\}$, share the same sign.
\end{proof}

Simple considerations (e.g.\ Examples \ref{ex.5pezzi} and
\ref{ex.ninv4p} below) show that the converse of Propositions
\ref{prop.inv} is not true in general. In order to partially invert
this proposition, different situations must be considered. We begin
with a simple consequence of Lemma \ref{thm.1}.

\begin{theorem}\label{thm:linnpezzi}
  Let $G:\R^k\to\R^k$ be a continuous strongly piecewise linear map as
  in Definition \ref{def.pwlin} with $\det(L_i)$ of constant sign for
  all $i=1, \ldots, n$. Assume also that there exists $q\in\R^k$ whose
  preimage $G^{-1}(q)$ is a singleton that belongs to at most two of
  the polyhedral cones $C_i$. Then $G$ is a Lipschitz
  homeomorphism.
\end{theorem}
\begin{proof}
  Lemma \ref{thm.1} and Remark \ref{rem:th1} imply that $\deg(G)=\pm
  1$. The assertion follows from Theorem \ref{th.4pr96} and Lemma
  \ref{lem.locglob}.
\end{proof}

\begin{remark}
  The condition in Theorem \ref{thm:linnpezzi} concerning the
  existence of a point $q$ whose preimage is a singleton belonging to
  at most two polyhedral cones, is equivalent to the existence of a
  half-line at the origin whose preimage is a single half-line. In
  fact, as a consequence of Theorem \ref{thm:linnpezzi}, one has that
  if the determinants $\det(L_i)$ have constant sign for all
  $i=1,\ldots, n$ the existence of such a half-line implies that all
  the half-lines at the origin must have the same property.
\end{remark}

\begin{remark}\label{rem:2p}
  Observe that the only nontrivial (i.e.\ such that are not reducible
  to linear maps) continuous strongly piecewise linear maps with $n=2$
  are those in which the cones are half-spaces. In fact, two linear
  endomorphisms of $\R^k$ that agree on two linearly independent
  vectors, necessarily coincide. Hence, when $n=2$, it is sufficient
  to consider the case when the two nontrivial cones are
  half-spaces. This has already been done in Proposition
  \ref{lem:lemalg}.
\end{remark}

The point $q$ in Theorem \ref{thm:linnpezzi} may be difficult to
determine if the linear maps $L_i$'s are given in a complicate
way. However, in some cases, invertibility of continuous nondegenerate
strongly piecewise linear maps can be deduced merely from their
nondifferentiability structure. The easiest nontrivial case, i.e.\
when $n=2$, has already been treated (Proposition \ref{lem:lemalg}) in
arbitrary dimension just by means of linear algebra. The other cases,
$n=3$ and $n=4$, will be investigated in dimension $k=2$ only.

\medskip We are now in a position to state our main result concerning
the invertibility of continuous strongly piecewise linear maps in
$\R^2$.

\begin{theorem}\label{th.main}
  Let $G:\R^2\to\R^2$ be as in Definition \ref{def.pwlin}. We have
  that, if one of the following conditions holds:
  \begin{enumerate}
  \item $n\in\{1,2,3\}$;
  \item $n=4$ and all the cones are convex;
  \end{enumerate}
  then $G$ has a continuous piecewise linear inverse.
\end{theorem}

Before we provide the proof of this result, we show with two examples
that the assumptions of Theorem \ref{th.main} are, to some extent,
sharp.

\begin{figure}[ht!]
  \definecolor{grigio0}{gray}{0.9} \definecolor{grigio1}{gray}{0.8}
  \definecolor{grigio2}{gray}{0.7} \definecolor{grigio3}{gray}{0.6}
  \definecolor{grigio4}{gray}{0.5}

  \centering
  \begin{pspicture}(-2.5, -3)(9.2,3)
    \uput{0}[0](3.3,0.5){\includegraphics[width=5.8cm,angle=0]{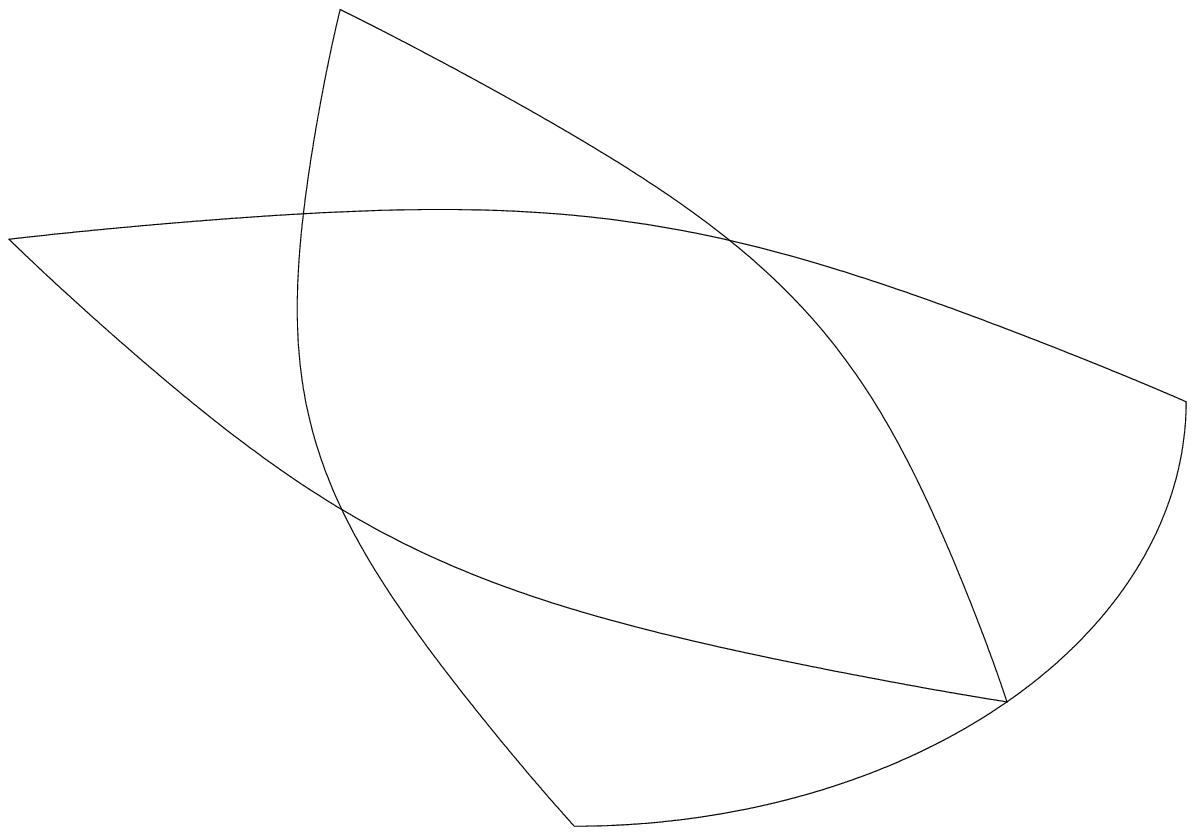}}
    \psline[linewidth=0.5pt]{->}(6.1, -1.8)(6.1, 2.8)
    \psline[linewidth=0.5pt]{->}(3.2, 0.55)(9.2, 0.55)
    \uput[u](3.45,1.35){$G(P_2)$} \uput[u](8.95,0.55){$G(P_1)$}
    \uput[dr](8.1,-0.8){$G(P_3)$} \uput[u](5,2.4){$G(P_4)$}
    \uput[dr](6.1,-1.3){$G(P_5)$}

    \pscurve[linewidth=0.8pt]{->}(2.2,-1.4)(3.4,-1.2)(4.3,-0.3)
    \uput[dr](3.4,-1.2){$G$}

    \psline[linewidth=0.5pt]{->}(-2.5, 0)(2.5, 0) %
    \psline[linewidth=0.5pt]{->}(0,-2.5)(0,2.5) \SpecialCoor
    \rput(0,0){ \psline[linestyle=dashed,linewidth=1.3pt](0,0)(2;67.5)
      \psline[linestyle=dashed,linewidth=1.3pt](0,0)(2;135)
      \psline[linestyle=dashed,linewidth=1.3pt](0,0)(2;202.5)
      \psline[linestyle=dashed,linewidth=1.3pt](0,0)(2;270)
      \psline[linestyle=dashed,linewidth=1.3pt](0,0)(2;0)
      \pscircle[linewidth=0.8pt](0,0){2} \uput[dr](0.6,-0.6){$C_5$}
      \uput[ur](0.8,0.4){$C_1$} \uput[u](-0.27,0.8){$C_2$}
      \uput[ul](-0.8,0){$C_3$} \uput[dl](-0.4,-0.7){$C_4$}

      \uput[ur](0.8,1.8){$P_2$} \uput[ur](2.1,0){$P_1$}
      \uput[u](-1.5,1.4){$P_3$} \uput[dl](-1.8,-0.8){$P_4$}
      \uput[dr](0,-2){$P_5$} }
  \end{pspicture}
  \caption{The image of $S^1$ under $G$ in Example
    \ref{ex.5pezzi}. For clarity's sake, the radial distance of
    $G(S^1)$ from $(0,0)$ has been rescaled.}\label{fig.ninv5p}
\end{figure}

Our first example shows that for $n>4$ there are $G$'s as above that
are not invertible even if the cones are convex.
\begin{example}\label{ex.5pezzi}
  Consider a nondegenerate continuous piecewise linear map $G\colon
  \R^2\to \R^2$ defined as in Definition \ref{def.pwlin} by
  \begin{gather*}
    L_1 =
    \begin{pmatrix}
      1 & -\sqrt{2} \\
      0 & \sqrt{2} - 1
    \end{pmatrix}
    \quad L_2 =
    \begin{pmatrix}
      -\sqrt{2} & -\sqrt{2}  + 1 \\
      1 & 0
    \end{pmatrix} \\
    L_3=
    \begin{pmatrix}
      0 & 1 \\
      -\sqrt{2} + 1 & - \sqrt{2}
    \end{pmatrix}
    \qquad L_4 =
    \begin{pmatrix}
      \sqrt{2} -1 & 0 \\
      -\sqrt{2} & 1
    \end{pmatrix}\quad
    L_5 =
    \begin{pmatrix}
      1 & 0 \\
      0 & 1
    \end{pmatrix}
  \end{gather*}
  where the corresponding cones are given, in polar coordinates, by
  the pairs $(\rho,\theta)$ with arbitrary $\rho$'s and $\theta$
  chosen as in the following table:
  \begin{center}
    \begin{tabular}{|c|c|c|c|c|}
      \hline
      \rule{0pt}{4mm}$C_1$ & $C_2$ & $C_3$ & $C_4$ & $C_5$\\
      \hline
      \rule[-2mm]{0pt}{7mm}
      $0\leq\theta\leq\frac{3}{8}\pi$ & 
      $\frac{3}{8}\pi\leq\theta\leq\frac{3}{4}\pi$ & $\frac{3}{4}\pi\leq\theta\leq \frac{9}{8}\pi$ &
      $\frac{9}{8}\pi\leq\theta\leq \frac{3}{2}\pi$ & $\frac{3}{2}\pi\leq\theta\leq 2\pi$ \\
      \hline
    \end{tabular}
  \end{center}
  This map is illustrated in Figure \ref{fig.ninv5p}. As the picture
  suggests, the above defined map $G$ is not invertible because it is
  not injective.
\end{example}

Our second example shows an instance of noninvertible $G$ with $n=4$
and one nonconvex cone.

\begin{example}\label{ex.ninv4p}
  Consider $G\colon \R^2\to \R^2$ as in Definition \ref{def.pwlin},
  with
  \[
  L_1 =
  \begin{pmatrix}
    1 & 0 \\
    0 & 1
  \end{pmatrix}, \quad L_2 =
  \begin{pmatrix}
    1 & 0 \\
    2\sqrt{3}& 1
  \end{pmatrix}, \quad L_3 =
  \begin{pmatrix}
    -2 & -\sqrt{3} \\
    -\sqrt{3} & -2
  \end{pmatrix}, \quad L_4=
  \begin{pmatrix}
    1 & 2\sqrt{3} \\ 0 & 1
  \end{pmatrix},
  \]
  and the cones are given, in polar coordinates, by the pairs
  $(\rho,\theta)$ with arbitrary $\rho$'s and $\theta$ chosen as in
  the following table:
  \begin{center}
    \begin{tabular}{|c|c|c|c|}
      \hline
      \rule{0pt}{4mm}$C_1$ & $C_2$ & $C_3$ & $C_4$\\
      \hline
      \rule[-2mm]{0pt}{7mm}
      $0\leq\theta\leq\frac{\pi}{2}$ & $\frac{\pi}{2}\leq\theta\leq\frac{2}{3}\pi$ & 
      $\frac{2}{3}\pi\leq\theta\leq\frac{11}{6}\pi$ & $\frac{11}{6}\pi\leq\theta\leq 2\pi$\\
      \hline
    \end{tabular}
  \end{center}
  Figure \ref{fig.ninv4p} illustrates this map. As the picture
  suggests, $G$ defined as above is not injective and therefore it is
  not invertible.
\end{example}

\begin{figure}[ht!]
  \centering
  \begin{pspicture}(-2.5, -3)(9.2,3)
    \uput{0}[0](3.55,0.85){\includegraphics[width=5.8cm,angle=0]{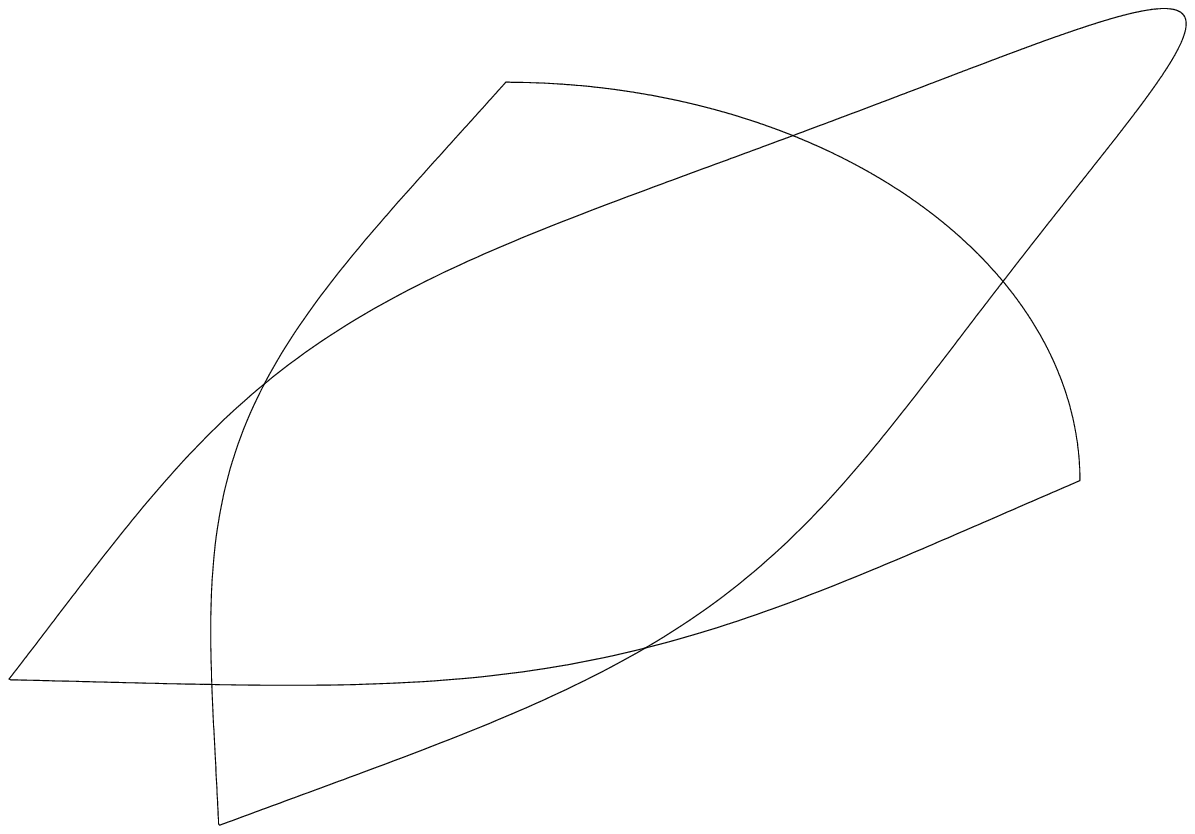}}
    \psline[linewidth=0.5pt]{->}(6, -1.8)(6, 2.8)
    \psline[linewidth=0.5pt]{->}(3.2, 0.55)(9.2, 0.55)
    \uput[ul](6,2.4){$G(P_2)$} \uput[dr](8.7,0.6){$G(P_1)$}
    \uput[d](4.7,-1.05){$G(P_3)$} \uput[dl](4,-0.3){$G(P_4)$}

    \pscurve[linewidth=0.8pt]{->}(2,-1.7)(3,-1.6)(4.2,-0.7)
    \uput[dr](3,-1.6){$G$}

    \psline[linewidth=0.5pt]{->}(-2.5, 0)(2.5, 0) %
    \psline[linewidth=0.5pt]{->}(0,-2.5)(0,2.5) \SpecialCoor
    \rput(0,0){ \psline[linestyle=dashed,linewidth=1.3pt](0,0)(2;120)
      \psline[linestyle=dashed,linewidth=1.3pt](0,0)(2;330)
      \psline[linestyle=dashed,linewidth=1.3pt](0,0)(2;90)
      \psline[linestyle=dashed,linewidth=1.3pt](0,0)(2;0)
      \pscircle[linewidth=0.8pt](0,0){2} \uput[ur](0.8,0.8){$C_1$}
      \uput[u](-0.35,1.2){$C_2$} \uput[dl](-0.8,-0.8){$C_3$}
      \uput[r](1.2,-0.4){$C_4$}

      \uput[ur](0,2){$P_2$} \uput[ur](2,0){$P_1$}
      \uput[ul](-1,1.55){$P_3$} \uput[dr](1.7,-1){$P_4$} }
  \end{pspicture}
  \centering
  \caption{The image of $S^1$ under $G$ in Example
    \ref{ex.ninv4p}. For clarity's sake, the radial distance of
    $G(S^1)$ from $(0,0)$ has been rescaled.}\label{fig.ninv4p}
\end{figure}

\medskip Let us now turn to the task of proving Theorem
\ref{th.main}. The proof is done differently according to the number
of nontrivial cones in which the plane is pie-sliced. The proof, in
the cases of $n=2$, boils down to Proposition \ref{lem:lemalg} whereas
the cases $n=3$ and $n=4$ will be treated with the help of Theorem
\ref{th.5pr96}. In order to apply this theorem it is necessary to
estimate the topological degree of our map $G$. This will be done by
the means of geometric considerations. The proof of the following
lemma is based on an elementary linear algebra argument and is left to
the reader.

\begin{lemma}\label{lem.stimangoli}
  Let $A \colon \R^2\to\R^2$ be linear and nonsingular and let
  $C\in\R^2$ be a cone with vertex at the origin.  Then
  $A(C)\subseteq\R^2$ is a cone with with vertex at the origin and the
  following statements hold:
  \begin{enumerate}
  \item If $C$ does not contain a half-plane, then $A(C)$ is strictly
    convex.
  \item If $C\varsubsetneq\R^2$ contains a half-plane, then so does
    $A(C)\varsubsetneq\R^2$.
  \end{enumerate}
\end{lemma}

This lemma has an useful consequence:

\begin{lemma}\label{lem.stimaint}
  Let $A \colon \R^2\to\R^2$ be linear and nonsingular and let
  $C\varsubsetneq\R^2$ be a cone with vertex at the origin.  Let
  $\Gamma$ be the image of the arc $S^1\cap C$. Then,
  \begin{equation}\label{eq.stimaint}
    \left|\int_{\Gamma} \omega\;\right|\in
    \begin{cases}
      [0,\pi)  & \text{if $C$ does not contain a half-plane,}\\
      [\pi,2\pi) & \text{otherwise.}
    \end{cases}
  \end{equation} 
  In particular, we have that $\left|\displaystyle\int_{\Gamma}
    \omega\;\right|<2\pi$ and $\left|\displaystyle\int_{\Gamma}
    \omega\;\right|<\pi$ when $C$ is strictly convex.
\end{lemma}
\begin{proof}
  Observe first that by Lemma \ref{lem.stimangoli} there exists a
  half-line $s$ starting at the origin that does not intersect
  $A(C)$. Clearly the differential form $\omega$ is exact in
  $\R^2\setminus s$.  Let $P_1$ and $P_2$ be the intersections of
  $\partial C$ with $S^1$. The path integral that appears in
  \eqref{eq.stimaint} does not depend on the chosen path connecting
  $A(p_1)$ and $A(p_2)$. With the choice of an appropriate path, for
  instance, the concatenation of a circular arc of radius
  $\abs{A(P_2)}$ from $A(P_2)$ with the radial segment through
  $A(P_1)$ (see Figure \ref{fig.intcurv}), it is not difficult to show
  that $|\int_{\Gamma} \omega |$ is merely the angular distance (we
  consider the angle that does not contain the half-line $s$) between
  $A(p_1)$ and $A(p_2)$ as seen from the origin. The assertion now
  follows from Lemma \ref{lem.stimangoli}.
\end{proof}

\begin{figure}[ht!]
  \centering
  \begin{pspicture}(-7.5, -2)(3.5,2)
    \rput(-5.3,0){ \pscircle[linewidth=0.3pt,linestyle=dashed](0,0){1}
      \pswedge[fillstyle=solid,fillcolor=lightgray,linestyle=none](0,0){1.45}{10}{160}
      \psarc[linewidth=0.3pt,linestyle=solid](0,0){1}{10}{160}
      \rput(0.984,0.173){\psdot}

      \uput[u](1.104,0.203){$P_1$} \rput(-0.939,0.342){\psdot}
      \uput[dl](-0.939,0.442){$P_2$} \SpecialCoor
      \psline[linewidth=0.5pt,linestyle=dashed](0, 0)(2 ; 10) %
      \psline[linewidth=0.5pt,linestyle=dashed](0,0)(2 ; 160)
      \psline[linewidth=0.5pt]{->}(-2, 0)(2, 0) %
      \psline[linewidth=0.5pt]{->}(0,-1.2)(0,2)
      \uput[ur](-0.5,1.4){$C$} \uput[dl](-0.4,-0.8){$S^1$} }
    \psline{->}(-4,-1.3)(-2.1,-1.3) \uput[d](-3,-1.3){$A$}
    \rput(-1.2,-1){ \psline[linewidth=0.5pt]{->}(-1, 0)(5, 0) %
      \psline[linewidth=0.5pt]{->}(0,-1)(0,3)
      \psarc[linewidth=0.7pt](0,0){2}{20}{80}
      \pscurve[linecolor=gray](0.347,1.969)(2.12,2.12)(3.758,1.368)
      \psline[linewidth=0.7pt,linestyle=dashed](0,0)(0.347,1.969)
      \psline[linewidth=0.7pt,linestyle=dashed](0,0)(3.758,1.368)
      \psline[linewidth=0.7pt,linestyle=solid](1.879,0.684)(3.758,1.368)
      \psline[linewidth=0.8pt,linestyle=dashed](0,0)(-1.1,2.2)
      \uput[u](-1,1){$s$} \uput[ur](3.758,1.348){$A(P_1)$}
      \rput(3.758,1.368){\psdot} \uput[u](0.487,1.969){$A(P_2)$}
      \rput(0.347,1.969){\psdot} \uput[ur](2.12,2.12){$\Gamma$}
      \psarc[linewidth=0.3pt,linestyle=dashed]{<->}(0,0){0.6}{20}{80}
      \uput[ur](0.34,0.45){\scriptsize $|\int_{\Gamma} \omega |$}
      \uput[dr](0.9,0.45){\scriptsize $|A(P_2)|$}
      \psline[linewidth=0.4pt,linestyle=solid,linecolor=gray]{<->}(1.5,1.4)(3,0.6)(2.9,1)
      \uput[d](3.4,0.6){\scriptsize integration path} }
  \end{pspicture}
  \caption{The integration path in Lemma
    \ref{lem.stimaint}}\label{fig.intcurv}
\end{figure}
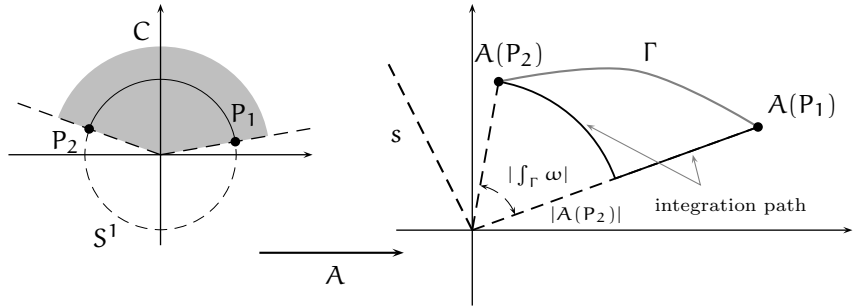

\begin{lemma}\label{lem.deg3p}
  Let $G$ be as in Theorem \ref{th.main} with $n=3$ and $\det L_i >0$,
  $\forall i=1, 2, 3$. Then, $\deg(G)=1$.
\end{lemma}
\begin{proof}
  We consider the two possible cases: when all the cones are strictly
  convex and when there is one cone containing a half-plane. For
  $i=1,2,3$, let $\Gamma_i$, be the image $G(S^1\cap C_i)$. In the
  first case, by Lemma \ref{lem.stimaint} we have that
  $\int_{\Gamma_i}\omega<\pi$ for $i=1,2,3$.  Hence, by Lemma
  \ref{thm.1} and formula \eqref{fintegrale},
  \begin{equation}\label{stimad3p}
    0<\deg(G)<\frac{\pi+\pi+\pi}{2\pi}=\frac{3}{2}.
  \end{equation}
  Which, the degree being an integer, implies $\deg(G)=1$.

  In the second case, only one of the cones, say $C_1$, may contain a
  half-plane. Thus, by Lemma \ref{lem.stimaint}, we have that
  $\int_{\Gamma_1}\omega<2\pi$ and $\int_{\Gamma_i}\omega<\pi$ for
  $i=2,3$. Hence, inequality \eqref{stimad3p} becomes
  \[
  0<\deg(G)<\frac{2\pi+\pi+\pi}{2\pi}=2.
  \]
  Which, again, implies $\deg(G)=1$.
\end{proof}

\begin{lemma}\label{lem.deg4p}
  Let $G$ be as in Theorem \ref{th.main} with $n=4$ and $\det L_i >0 $
  $\forall i=1, 2, 3, 4$. Then $\deg(G)=1$.
\end{lemma}
\begin{proof}
  For $i=1,\ldots,4$, let $\Gamma_i$, be the image $G(S^1\cap
  C_i)$. By Lemma \ref{lem.stimaint} we have that
  $\int_{\Gamma_i}\omega<\pi$ for $i=1,\ldots,4$ since all the cones
  are convex. Hence, by Lemma \ref{thm.1} and formula
  \eqref{fintegrale},
  \[
  0<\deg(G)<\frac{\pi+\pi+\pi+\pi}{2\pi}=2.
  \]
  Which, the degree being an integer, implies $\deg(G)=1$.
\end{proof}
\begin{remark}
  If $\det L_i <0$ $\forall i=1, \ldots, n$, composing $G$ with the
  linear maps whose matrix in the standard basis of $\R^2$ is $J =
  \begin{pmatrix}
    0 & 1 \\
    1 & 0
  \end{pmatrix}
  $, we get that $JG \colon \R^2 \to \R^2$ is one--to--one so that $G$
  is invertible as well.
\end{remark}
We are now in a position to prove our main theorem.

\begin{proof}[Proof of Theorem \ref{th.main}]

  \noindent{\bf(Case $n=1$.)} In this case $G$ is linear with nonzero
  determinant. Thus, there is nothing to prove.

  \noindent{\bf(Case $n=2$.)} See Proposition \ref{lem:lemalg} and
  Remark \ref{rem:2p}.

  \noindent{\bf(Cases $n=3$ and $n=4$.)} In both cases, it follows
  from Lemmas \ref{lem.deg3p} and \ref{lem.deg4p} that $\deg(G)=1$,
  then the assertion follows from Theorem \ref{th.4pr96}.
\end{proof}

\begin{example}
  Let us consider a decomposition of the space $\R^3$ into four convex
  wedges $C_1,\ldots,C_4$ with a common edge along the straight line
  $r$. Let $G\colon\R^3\to\R^3$ be a continuous strongly piecewise
  linear map with respect to this decomposition and with
  \[
  G(x) = L_i x, \qquad x \in C_i,\; i=1,\ldots,4
  \]
  and assume that $\det L_i$ share the same sign for
  $i=1,\ldots,4$. Then, as a consequence of Theorem \ref{th.main}, we
  have that $G$ is invertible with continuous strongly piecewise
  linear inverse. To see that, consider a plane $P$ orthogonal to
  $r$. Clearly, the restriction $G|_P$ is invertible by Theorem
  \ref{th.main}. Similarly, since $G(x_r)=L_1x_r=\ldots=L_4x_r$, for
  any point $x_r\in r$, and the $L_i$'s are isomorphisms, $G$ is
  invertible on $r$. Given any vector $y\in\R^3$, we can obtain
  $G^{-1}(y)$ by the following argument. Write $y=y_P+y_r$ where $y_P$
  and $y_r$ denote the orthogonal projections of $y$ onto $P$ and $r$,
  respectively. Then one has
  \[
  G^{-1}(y)=L_1^{-1}(y)+(G|_P)^{-1}(y).
  \]
\end{example}

We conclude this section with an example showing that our main result
(at least when $n=4$) cannot be deduced from the well-known Clarke's
Theorem. In fact, we exhibit a continuous piecewise linear map which
is invertible by our main result although it does not satisfy the
assumptions of Clarke's Theorem.
\begin{example}\label{ex.5pezzicl}
  Consider a continuous piecewise linear map $G\colon \R^2\to \R^2$
  defined as in Definition \ref{def.pwlin} by
  \[
  L_1 =
  \begin{pmatrix}
    1 & 0 \\
    0 & 1
  \end{pmatrix}, \quad L_2 =
  \begin{pmatrix}
    \frac{1}{10} & 0 \\[1mm]
    -10& 1
  \end{pmatrix}, \quad L_3 =
  \begin{pmatrix}
    \phantom{+}\frac{5}{100} & \phantom{+}\frac{5}{100} \\[2mm]
    -\frac{455}{100} & -\frac{445}{100}
  \end{pmatrix}, \quad L_4=
  \begin{pmatrix}
    1 & 1 \\[1mm]
    0 & \frac{1}{10}
  \end{pmatrix},
  \]
  where the corresponding cones are the quadrants I,\ldots,IV,
  respectively. This map is illustrated in Figure
  \ref{fig.noclarke4p}. As the picture suggests, $G$ has degree
  $1$. (The map $G$ has been found with the help of a short FORTRAN
  program that assisted us in sifting many potential examples.)
\end{example}

\begin{figure}[ht!]
  \centering
  \begin{pspicture}(-2, -3)(9.7,3)
    \rput(6.2,1){\includegraphics[width=5.8cm,angle=0]{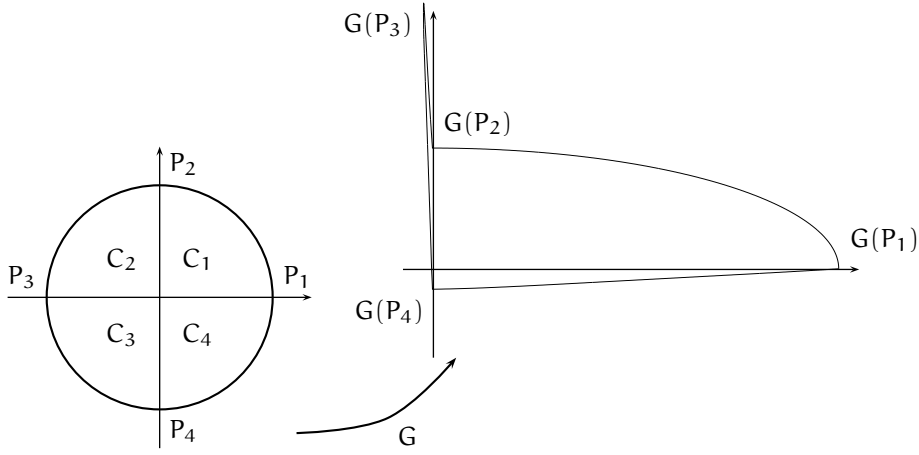}}
    \psline[linewidth=0.5pt]{->}(3.6, -1.8)(3.6, 2.8)
    \psline[linewidth=0.5pt]{->}(3.2, -0.63)(9.2, -0.63)
    \uput[ur](8.95,-0.55){$G(P_1)$} \uput[ur](3.6,1){$G(P_2)$}
    \uput[dl](3.45,2.9){$G(P_3)$} \uput[dl](3.6,-0.9){$G(P_4)$}

    \pscurve[linewidth=0.8pt]{->}(1.8,-2.8)(3,-2.6)(3.9,-1.8)
    \uput[dr](3,-2.6){$G$}

    \rput(0,-1){ \psline[linewidth=0.5pt]{->}(-2, 0)(2, 0) %
      \psline[linewidth=0.5pt]{->}(0,-2)(0,2)
      \pscircle[linewidth=0.8pt](0,0){1.5} \rput(0.5,0.5){$C_1$}
      \rput(-0.5,0.5){$C_2$} \rput(-0.5,-0.5){$C_3$}
      \rput(0.5,-0.5){$C_4$}

      \uput[ur](0,1.5){$P_2$} \uput[ur](1.5,0){$P_1$}
      \uput[ul](-1.5,0){$P_3$} \uput[dr](0,-1.5){$P_4$} }
  \end{pspicture}
%
%
%
  \caption{The image of $S^1$ under $G$ in Example
    \ref{ex.5pezzicl}. For clarity's sake, the radial distance of
    $G(S^1)$ from $(0,0)$ and the proportion between axes have been
    both altered.}\label{fig.noclarke4p}
\end{figure}


\section{Application: Piecewise differentiable functions}

We now provide some applications of the results of the previous
section to the local invertibility of $PC^1$ functions. The basis for
our considerations is the following consequence of Theorem
\ref{th.4pr96}.

\begin{theorem}\label{th.conpunte}
  Let $f$ be an $\R^k$-valued $PC^1$ function in a neighborhood of
  $x_0\in\R^k$. Assume that
  \begin{enumerate}
  \item All the elements of the geralized Jacobian of $f$ at $x_0$
    have the same sign;
  \item The Bouligand differential of $f$ at $x_0$ is an invertible
    piecewise linear map.
  \end{enumerate}
  Then $f$ is locally invertible at $x_0$.
\end{theorem}

\begin{proof}
  It is not difficult to show that since $f'(x_0,\cdot)$ is
  invertible,
  \[
  \deg\big(f'(x_0,\cdot),V,0\big)=\mathfrak{s},
  \]
  where $\mathfrak{s}$ denotes the common sign of the elements of the
  generalized Jacobian of $f$ at $x_0$

  We claim that in a sufficiently small neighborhood $V$ of $x_0$ the
  map $f$ is admissibly homotopic to $f'(x_0,\cdot)$ so that, by
  homotopy invariance, $\deg(f,V,0)=\mathfrak{s}$.  The assertion
  follows from Theorem \ref{th.4pr96}.

  We now prove the claim. Consider the homotopy
  \[
  H(x,\lambda)=f(x-x_0)+\lambda\abs{x-x_0}\varepsilon(x-x_0),\quad
  \lambda\in [0,1],
  \]
  and let $\{i=1,\ldots,k\}$ be the active index set of $f$ at
  $x_0$. Observe that
  \[
  m := \inf\big\{\abs{f(v)} \colon \abs{v}=1\big\}
  =\min_{i=1,\ldots,k}\|df_i\|>0.
  \]
  Thus,
  \[
  \abs{H(x,\lambda)} \geq \big( m -
  \abs{\varepsilon(x-x_0)}\big)\abs{x-x_0}.
  \]
  This shows that in a conveniently small ball centered at $x_0$,
  homotopy $H$ is admissible.
\end{proof}

\begin{example}
  Let $R_1=\{(x,y)\in\R^2:y>x^2\}$, $R_2=\{(x,y)\in\R^2:y<-x^2\}$, and
  $R_3=\R^2\setminus (R_1\cup R_2)$.  Consider the $PC^1$ map
  $f\colon\R^2\to\R^2$ given by
  \[
  f(x,y)=\begin{cases}
    (x,2y-x^2) & \text{for $(x,y)\in R_1$},\\
    (x,2y+x^2) & \text{for $(x,y)\in R_2$},\\
    (x,y) & \text{for $(x,y)\in R_3$}.
  \end{cases}
  \]
  One has that $f'\big((0,0),\cdot)$ is the identity and that the
  generalized Jacobian at the origin consists of matrices which have
  positive determinant. Hence, by Theorem \ref{th.conpunte}, $f$ is
  locally invertible about the origin.
\end{example}

In order to apply Theorem \ref{th.conpunte} above one needs to know
when the linearization of a $PC^1$ map (which is a continuous
piecewise linear map) is invertible. This is what all the previous
section is about. Criteria for the local invertibility of $PC^1$ map
will be deduced from Theorem \ref{th.conpunte} combined with the
results of the previous section.

Let $f$ be an $\R^k$-valued $PC^1$ function in a sufficiently small
ball $B(x_0,\rho)\subseteq\R^k$, and let
$\mathcal{I}_0=\{1,\ldots,n\}$ be the active index set in
$B(x_0,\rho)$. For each $i\in\mathcal{I}_0$ define
\begin{equation}\label{eq:fpez}
  S_i := \big\{ x\in B(x_0,\rho):f(x)=f_i(x)\big\}.
\end{equation}
Let $C_1,\ldots,C_n$ be the tangent cones (in the sense of Bouligand)
at $x_0$ to the sectors $S_1,\ldots,S_n$. Assume that for
$i,j\in\{1,\ldots,n\}$
\[
\ud f_i(x_0)x = \ud f_j(x_0)x\quad\text{for any $x \in C_i \cap C_j$,}
\]
and define
\begin{equation}\label{eq:Flin}
  F(x)= \ud f_i(x_0)x \qquad x\in C_i,\; i=1,\ldots,n
\end{equation}
so that $F$ is a continuous piecewise linear map (compare
\cite{KS94}).

This section is concerned with local invertibility of such maps. We
first consider arbitrary dimension.

\begin{corollary} \label{thm:topo} Let $f$ and $F$ be as in
  \eqref{eq:fpez}-\eqref{eq:Flin}, with $F$ nondegenerate at $0$.
  Assume also that there exists $p\in\R^k$ whose preimage under $F$,
  $F^{-1}(p)$, is a singleton that belongs to at most two of the cones
  $C_i$. Then $f$ is a Lipschitzian homeomorphism in a sufficiently
  small neighborhood of $x_0$.
\end{corollary}
  
\begin{proof}
  From Theorem \ref{thm:linnpezzi} it follows that $F$ is
  invertible. The assertion follows from Theorem \ref{th.conpunte}.
\end{proof}

The above Theorem \ref{thm:topo} can be greatly simplified when the
number of cones is $n=2$, in the sense that the assumption on the
existence of the special point $p$ can be dropped altogether.  In
fact, in dimension $k=2$ this is true also for $k=3$ and, when $k=4$,
one can replace it by merely requiring the convexity of the tangent
cones to the sectors.

\begin{corollary}
  Let $f$ and $F$ be as in \eqref{eq:fpez}-\eqref{eq:Flin}, with $F$
  nondegenerate at $0$ and \mbox{$n=2$}. Then $f$ is a Lipschitzian
  homeomorphism in a sufficiently small neighborhood of $x_0$.
\end{corollary}

\begin{proof}
  From Proposition \ref{lem:lemalg} it follows that $F$ is
  invertible. The assertion follows from Theorem \ref{th.conpunte}.
\end{proof}

We finally consider dimension $k=2$ of the ambient space.

\begin{corollary} \label{thm:topoGeo} Let $f$ and $F$ be as in
  \eqref{eq:fpez}-\eqref{eq:Flin}, with $F$ nondegenerate at $0$.  We
  have that if either $n\in\{1,2,3\}$ or $n=4$ and all the cones
  $C_i$'s are convex, then $f$ is a Lipschitzian homeomorphism in a
  sufficiently small neighborhood of $x_0$.
\end{corollary}
  
\begin{proof}
  Since $F$ is nondegenerate then it is invertible by Theorem
  \ref{th.main}. Theorem \ref{th.conpunte}, yields the assertion.
\end{proof}


\end{document}